\documentclass{amsart}

\usepackage{amsmath,amssymb}
\usepackage[colorlinks=true,linkcolor=blue,citecolor=blue,urlcolor=blue]{hyperref}
\usepackage[skip=0.5\baselineskip plus 2pt minus 1pt]{parskip}  

\newtheorem{theorem}{Theorem}[section]
\newtheorem{lemma}[theorem]{Lemma}
\newtheorem{corollary}[theorem]{Corollary}
\newtheorem*{standing}{Standing assumption}
\theoremstyle{definition}
\newtheorem{example}[theorem]{Example}
\theoremstyle{remark}
\newtheorem{remark}[theorem]{Remark}

\numberwithin{equation}{section}

\DeclareMathOperator{\Sol}{Sol}
\DeclareMathOperator{\PVI}{PVI}
\DeclareMathOperator{\PCP}{PCP}
\DeclareMathOperator{\LCP}{LCP}
\DeclareMathOperator{\Stat}{Stat}
\DeclareMathOperator{\rank}{rank}

\newcommand{\R}{\mathbb{R}}

\begin{document}

\title{Dimension in polynomial variational inequalities}

\author{Vu Trung Hieu}
\address{LIP6, Sorbonne Universit\'e, Paris, France}
\email{trung-hieu.vu@lip6.fr, hieuvut@gmail.com}



\begin{abstract}
We study the dimension of the solution set of a variational inequality whose operator is polynomial and whose constraint set is defined by finitely many convex polynomial inequalities and affine equations. Under the Abadie constraint qualification, we bound this dimension from above in terms of the ranks of the Jacobian matrices of the associated Karush--Kuhn--Tucker maps; the bound is attained in nontrivial examples and yields a sufficient condition for finiteness of the solution set. For polynomial complementarity problems, the bound is expressed through principal submatrices of the Jacobian matrix of the operator, and the resulting finiteness criterion extends the classical nondegeneracy condition for linear complementarity problems.
\end{abstract}

\maketitle

\section{Introduction}
We consider the variational inequality
\begin{equation}\label{VI}
\text{find } x\in K \ \text{ such that } \ \langle F(x),y-x\rangle\geq 0 \quad \forall y\in K,
\end{equation}
where $K\subset\R^n$ is a nonempty closed convex semialgebraic set, $F:\R^n\to\R^n$ is a polynomial map, and $\langle \cdot,\cdot\rangle$ is the usual scalar product of $\R^n$. We call \eqref{VI} a \emph{polynomial variational inequality} and denote the problem and its solution set by $\PVI(K,F)$ and $\Sol(K,F)$, respectively. The class of polynomial variational inequalities contains, as special cases, linear complementarity problems, affine variational inequalities, tensor complementarity problems, and polynomial complementarity problems (see, e.g., \cite{CPS1992,Gowda16,LTY05,SongQi2015}). The existence and stability of solutions of polynomial variational inequalities and polynomial complementarity problems have been studied, for instance, in \cite{Hieu19,HWY20}.

By the Tarski--Seidenberg theorem with quantifiers (see, e.g., \cite[Theorem 2.6]{Coste02} and \cite{HieuMTA20}), the solution set $\Sol(K,F)$ is a semialgebraic subset of $\R^n$. Consequently, its dimension is well defined: it is an integer between $0$ and $n$ if $\Sol(K,F)\neq\emptyset$, and it equals $-\infty$ by convention if $\Sol(K,F)=\emptyset$. The present paper is devoted to the study of this dimension for constraint sets $K$ described by finitely many convex polynomial inequalities and affine equations (see Section~\ref{sec:Pre}).

Our main result, Theorem~\ref{thm:dim1}, bounds $\dim(\Sol(K,F))$ from above, under the Abadie constraint qualification, in terms of the ranks of the Jacobian matrices of the Karush--Kuhn--Tucker (KKT) maps attached to the pseudo-faces of $K$ (see Section~\ref{sec:3}); it also yields a sufficient condition for the solution set to be finite. The bound is attained in nontrivial examples (Example~\ref{ex:1} and Remark~\ref{rmk:nonsing}). The main tool is Lemma~\ref{lem:dim}, which bounds the dimension of a semialgebraic subset of a level set of a polynomial map in terms of the minimal rank of the Jacobian matrix of the map on that subset.

For polynomial complementarity problems, we compute the rank of the Jacobian matrix of the KKT map explicitly: for the pseudo-face where exactly the coordinates indexed by $\alpha$ vanish, this rank equals, at every point, $2|\alpha|$ plus the rank of the principal submatrix of $J_F$ indexed by the complementary set $[n]\setminus\alpha$ (Lemma~\ref{lem:pcp}). This yields a bound for the dimension of the solution set and a finiteness criterion in terms of the principal minors of $J_F$ on the nonnegative orthant (Corollary~\ref{thm:dim2}); for linear complementarity problems, this criterion is the classical nondegeneracy condition. Finally, Remark~\ref{rmk:frac} observes that these results also apply to the set of stationary points of a polynomial fractional program \cite{TTK04}.

The paper is organized as follows. Section~\ref{sec:Pre} collects the necessary facts from semialgebraic geometry and from the theory of variational inequalities. Section~\ref{sec:3} contains the main result, and Section~\ref{sec:4} is devoted to polynomial complementarity problems.

\section{Preliminaries}\label{sec:Pre}
This section recalls the definitions, notation, and auxiliary results from semialgebraic geometry and from the theory of variational inequalities that are used throughout the paper.

\subsection{Semialgebraic sets and dimension}

A subset of $\R^n$ is \textit{semialgebraic} \cite[Definition 2.1.4]{BCF98} if it is a finite union of sets of the form
\begin{equation*}
\big\{x\in \R^n\,:\,f_1(x)=\dots=f_\ell(x)=0,\ g_{\ell+1}(x)<0,\dots,g_m(x)<0\big\},
\end{equation*}
where $f_1,\dots, f_\ell, g_{\ell+1},\dots,g_m$ are polynomials with real coefficients. The class of semialgebraic sets is closed under finite unions, finite intersections, and complements. The Tarski--Seidenberg theorem \cite[Theorem 2.3]{Coste02} asserts that the image of a semialgebraic set under a linear projection is semialgebraic. Equivalently \cite[Theorem 2.6]{Coste02}, every set defined by a first-order formula---that is, a formula built from polynomial equalities and inequalities by means of the logical connectives and the quantifiers $\forall$ and $\exists$ over real variables---is semialgebraic.

Let $S_1\subset \R^N$ and $S_2\subset \R^p$ be semialgebraic sets. A map $G:S_1\to S_2$ is said to be \textit{semialgebraic} \cite[Definition 2.2.5]{BCF98} if its graph $\{(x,G(x)):x\in S_1\}$ is a semialgebraic subset of $\R^N\times\R^p$. Restrictions of polynomial maps and of linear maps to semialgebraic sets are semialgebraic, and by the Tarski--Seidenberg theorem the image of a semialgebraic set under a semialgebraic map is semialgebraic.

Every semialgebraic set $S\subset \R^N$ admits a decomposition \cite[Theorem 2.3.6]{BCF98}
$$S=C_1 \cup \dots \cup C_s$$
into finitely many pairwise disjoint semialgebraic sets $C_i$, called \emph{cells}, each of which is semialgebraically homeomorphic to an open cube $(0,1)^{d_i}\subset\R^{d_i}$ (with the convention that $(0,1)^{0}$ is a point). The \textit{dimension} of $S$ is
$$\dim(S):=\max\{d_1,\dots,d_s\}.$$
This number coincides with the dimension defined in \cite[Section 2.8]{BCF98} (this follows from \cite[Proposition 2.8.5(i) and Theorem 2.8.8]{BCF98}); in particular, it does not depend on the decomposition. We adopt the convention $\dim(\emptyset):=-\infty$. A nonempty semialgebraic set has dimension $0$ if and only if it is finite.

We shall use the following properties, in which $S$, $S_1$, $S_2$, $\dots$ are semialgebraic sets:
\begin{itemize}
\item[(D1)] if $S_1\subset S_2$, then $\dim(S_1)\leq \dim(S_2)$;
\item[(D2)] $\dim(S_1\cup \dots\cup S_k)=\max\{\dim(S_1),\dots,\dim(S_k)\}$ \cite[Proposition 2.8.5(i)]{BCF98};
\item[(D3)] if $G:S\to\R^p$ is a semialgebraic map, then $\dim(G(S))\leq\dim(S)$, with equality when $G$ is injective \cite[Theorem 2.8.8]{BCF98};
\item[(D4)] a semialgebraic set $S\subset\R^n$ satisfies $\dim(S)=n$ if and only if its interior is nonempty \cite[Proposition 2.8.4]{BCF98}.
\end{itemize}
Property (D1) follows from (D2) applied to $S_2=S_1\cup(S_2\setminus S_1)$. We also note the following consequence of the definition of dimension, (D1), and (D3): \emph{if $C\subset\R^N$ is a cell of dimension $d$ and $O\subset\R^N$ is an open semialgebraic set (for instance, an open ball) meeting $C$, then $\dim(C\cap O)=d$.} Indeed, if $h:C\to(0,1)^d$ is a semialgebraic homeomorphism, then $h(C\cap O)$ is a nonempty open subset of $(0,1)^d$; it therefore contains a small open cube, which is affinely homeomorphic to $(0,1)^d$.

Throughout the paper, $J_P(z)$ denotes the Jacobian matrix of a differentiable map $P$ at $z$.

\subsection{Polynomial variational inequalities}

Let $K$ and $F$ be as in \eqref{VI}. Recall that $x\in\Sol(K,F)$ if and only if $-F(x)\in N_K(x)$, where
$$
N_K(x)=\{x^*\in \R^n:\langle x^*,y-x\rangle\leq 0 \ \ \forall y\in K\}
$$
is the normal cone to $K$ at $x\in K$. If $x$ is an interior point of $K$, then $N_K(x)=\{0\}$; hence such a point $x$ belongs to $\Sol(K,F)$ if and only if $F(x)=0$.

\begin{standing}
Throughout the paper, $K$ is a nonempty set of the form
\begin{equation}\label{K_set}
K = \big\lbrace x \in \R^n:g_i(x) \leq 0, \ i\in [m], \ h_j(x) = 0, \ j \in [\ell]\big\rbrace,
\end{equation}
where $g_i$, $i\in [m]:=\{1,\dots,m\}$, are convex polynomials and $h_j$, $j\in [\ell]$, are affine functions.
\end{standing}
Every set of the form \eqref{K_set} is closed, convex, and semialgebraic. We write $g=(g_1,\dots,g_m)$, $h=(h_1,\dots,h_\ell)$, and $g_\alpha=(g_i)_{i\in\alpha}$ for $\alpha\subset[m]$.

With every index set $\alpha\subset [m]$ we associate the \textit{pseudo-face}
$$
K_{\alpha}=\left\lbrace x\in \R^n:g_i(x) =0 \ (i\in\alpha),\ g_i(x) <0 \ (i\in [m]\setminus\alpha),\ h_j(x) = 0\ (j\in [\ell])\right\rbrace,
$$
which may be empty. The pseudo-faces are semialgebraic and pairwise disjoint, and
\begin{equation}\label{K}
K=\bigcup_{\alpha\subset [m]} K_{\alpha}.
\end{equation}
Since $\Sol(K,F)\subset K$, the decomposition \eqref{K} yields
\begin{equation}\label{p_face}
\Sol(K,F)=\bigcup_{\alpha\subset [m]}\left[ \Sol(K,F)\cap K_{\alpha}\right].
\end{equation}
The solution set is defined by the first-order formula
$$x\in K\ \wedge\ \forall y\,\big(y\in K\Rightarrow \langle F(x),y-x\rangle\geq0\big),$$
so it is semialgebraic by \cite[Theorem 2.6]{Coste02}; hence so is each set $\Sol(K,F)\cap K_{\alpha}$, and \eqref{p_face} together with (D2) gives
\begin{equation}\label{dimsum1}
\dim(\Sol(K,F))=\max_{\alpha\subset [m]}\dim(\Sol(K,F)\cap K_{\alpha}).
\end{equation}

The Bouligand--Severi tangent cone to $K$ at $x \in K$ (see, e.g., \cite[p.~15]{FaPa03}), denoted by $T_K(x)$, consists of the vectors $v \in\R^n$ for which there exist a sequence $\{y^k\}\subset K$ and a sequence of positive scalars $\{t^k\}$ such that
$$\lim_{k\to\infty}y^k=x, \ \; \lim_{k\to\infty}t^k=0,\; \text{ and } \lim_{k\to\infty}\dfrac{y^k-x}{t^k}=v.$$
The linearization cone to $K$ at $x\in K$ (see, e.g., \cite[p.~17]{FaPa03}) is
$$L_K(x)=\big\lbrace v\in\R^n:\langle\nabla g_i(x),v\rangle \leq 0\ (i\in I(x)),\ \langle\nabla h_j(x),v\rangle=0\ (j\in [\ell])\big\rbrace,$$
where $I(x) := \{i\in [m] : g_i(x) = 0\}$ is the active index set at $x$. The set $K$ satisfies the \textit{Abadie constraint qualification} (ACQ) at $x\in K$ \cite[p.~17]{FaPa03} if $L_{K}(x)=T_{K}(x)$; we say that the ACQ \emph{holds on} $K$ if it holds at every point of $K$. If the ACQ holds on $K$, then the Karush--Kuhn--Tucker conditions characterize the solutions of variational inequalities over $K$ \cite[Proposition 1.3.4]{FaPa03}: $x\in\Sol(K,F)$ if and only if there exist $\lambda\in \R^{m}$ and $\mu\in \R^{\ell}$ such that
\begin{equation}\label{KKT}
\begin{array}{c}
F(x)+\displaystyle\sum_{i\in [m]}\lambda_i\nabla g_i(x)+\displaystyle\sum_{j\in [\ell]}\mu_j\nabla h_j(x)=0,\\[2ex]
\lambda^T g(x)=0,\quad \lambda\geq 0, \quad g(x) \leq 0,\quad h(x)=0.
\end{array}
\end{equation}

The ACQ holds at $x\in K$, for instance, if the gradients $\nabla g_i(x)$, $i\in I(x)$, and $\nabla h_j(x)$, $j\in[\ell]$, are linearly independent (the linear independence constraint qualification; see, e.g., \cite[p.~17]{FaPa03}). It also holds at every point of $K$ whenever all the functions $g_i$ and $h_j$ are affine (in which case $K$ is polyhedral).

\section{Upper bounds for dimensions}\label{sec:3}

This section bounds the dimension of the solution set of a polynomial variational inequality from above. The bounds are expressed through the ranks of the Jacobians of the KKT maps attached to the pseudo-faces of $K$. The main tool is the following lemma on the dimension of level sets of polynomial maps.

\begin{lemma}\label{lem:dim}
Let $P:\R^N\to\R^{p}$ be a polynomial map, let $v\in\R^{p}$, and let $W$ be a nonempty semialgebraic subset of the level set $Z:=P^{-1}(v)$. Then
\begin{equation}\label{dimW}
\dim(W)\leq N-\min_{z\in W}\rank(J_P(z)).
\end{equation}
\end{lemma}

The minimum in \eqref{dimW} exists because the rank takes only finitely many values. Lemma~\ref{lem:dim} is a quantitative form of the elementary fact that, near a point where $J_P$ has rank $k$, the level set $Z$ lies in a $C^\infty$ submanifold of codimension $k$. For a related result, namely the existence of a stratification of a closed semialgebraic set into smooth strata on each of which a given semialgebraic map has constant rank, see \cite[4.8]{vdDM96}. Since we could not find a reference with the exact formulation of Lemma~\ref{lem:dim}, we include a proof for completeness.

\begin{proof}
\emph{Step~1: a local bound.} We show that \emph{for every $z\in Z$ there exists an open ball $B$ centered at $z$ such that}
\begin{equation}\label{local_bound}
\dim(W\cap B)\leq N-\rank(J_P(z)).
\end{equation}
Fix $z\in Z$ and put $k:=\rank(J_P(z))$. If $k=0$, then \eqref{local_bound} reads $\dim(W\cap B)\leq N$, which holds for every ball $B$ by (D1) and (D4), since $\dim(\R^N)=N$. Assume now that $k\geq1$. Since the matrix $J_P(z)$ has rank $k$, it has $k$ linearly independent rows; that is, there are indices $i_1<\dots<i_k$ such that the gradients $\nabla P_{i_1}(z),\dots,\nabla P_{i_k}(z)$ are linearly independent. Put $\widetilde P:=(P_{i_1},\dots,P_{i_k}):\R^N\to\R^k$. The $k\times N$ matrix $J_{\widetilde P}(z)$ has rank $k$, so it has $k$ linearly independent columns. A permutation of the coordinates of $\R^N$ is a linear isomorphism; it preserves dimensions by (D3) and does not change the ranks of the Jacobian matrices involved. Hence we may assume that these columns are the last $k$ columns. We write $y=(u,w)\in\R^{N-k}\times\R^k$ and $z=(u_0,w_0)$. Then the $k\times k$ matrix of partial derivatives of $\widetilde P$ with respect to $w$ is invertible at $z$. By the implicit function theorem, applied to the equation $\widetilde P(u,w)=\widetilde P(z)$, there exist an open neighborhood $O$ of $z$, an open neighborhood $V$ of $u_0$ in $\R^{N-k}$, and a $C^\infty$ map $\varphi:V\to\R^k$ such that
$$\{y\in O:\ \widetilde P(y)=\widetilde P(z)\}=\{(u,\varphi(u)):\ u\in V\}.$$
Let $B\subset O$ be an open ball centered at $z$, and let $\pi:\R^N\to\R^{N-k}$ be the projection $\pi(u,w)=u$. Every $y\in Z$ satisfies $P(y)=v=P(z)$, hence $\widetilde P(y)=\widetilde P(z)$. Therefore every point of $Z\cap B$ is of the form $(u,\varphi(u))$ with $u\in V$, and consequently $\pi$ is injective on $Z\cap B$: if $y,y'\in Z\cap B$ and $\pi(y)=\pi(y')=u$, then $y=(u,\varphi(u))=y'$. The set $W\cap B$ is semialgebraic and, since $W\subset Z$, the restriction of $\pi$ to $W\cap B$ is an injective semialgebraic map. By (D3), $\dim(W\cap B)=\dim(\pi(W\cap B))$; since $\pi(W\cap B)\subset\R^{N-k}$, (D1) and (D4) give $\dim(\pi(W\cap B))\leq N-k$. This proves \eqref{local_bound}.

\emph{Step~2: from the local bound to \eqref{dimW}.} Put $d:=\dim(W)\geq0$. By the definition of dimension, there is a decomposition of $W$ into cells one of which, say $C$, has dimension $d$. Choose a point $z\in C$ and an open ball $B$ centered at $z$ as in Step~1. Since $z\in C\cap B$, the ball $B$ meets $C$, so the remark following (D4) in Section~\ref{sec:Pre} gives $\dim(C\cap B)=d$. Since $C\cap B\subset W\cap B$, property (D1) and \eqref{local_bound} yield
$$d=\dim(C\cap B)\leq\dim(W\cap B)\leq N-\rank(J_P(z))\leq N-\min_{z'\in W}\rank(J_P(z')),$$
where the last inequality holds because $z\in W$. This proves \eqref{dimW}.
\end{proof}

Let $\alpha\subset [m]$ and put $N_\alpha:=n+|\alpha|+\ell$. Consider the polynomial map $\Phi_{\alpha}:\R^{n}\times \R^{|\alpha|}\times \R^{\ell}\to\R^{N_\alpha}$ defined by
$$\Phi_{\alpha}(x,\lambda_\alpha,\mu)=\Big( F(x)+\displaystyle\sum_{i\in \alpha}\lambda_i\nabla g_i(x)+\displaystyle\sum_{j\in [\ell]}\mu_j\nabla h_j(x),\, g_\alpha(x),\, h(x)\Big) ^T,$$
where $\lambda_{\alpha}=(\lambda_{i})_{i\in\alpha}$. We call $\Phi_\alpha$ the \emph{KKT map} of $\PVI(K,F)$ on the pseudo-face $K_\alpha$. Its zero set
$$Z_\alpha:=\Phi_{\alpha}^{-1}(0)\subset\R^{N_\alpha}$$
is semialgebraic. The Jacobian matrix of $\Phi_\alpha$ is the $N_\alpha\times N_\alpha$ polynomial matrix
\begin{equation}\label{Dphi}
J_{\Phi_\alpha}(x,\lambda_{\alpha},\mu)=\begin{bmatrix}
J_F(x)+\displaystyle\sum_{i\in \alpha}\lambda_i\nabla^2 g_i(x) &  \nabla g_{\alpha}(x) & \nabla h(x)\\
\nabla g_{\alpha}(x)^T & 0 & 0\\
\nabla h(x)^T & 0 & 0
\end{bmatrix},
\end{equation}
where $\nabla^2 g_i$ is the Hessian of $g_i$, and $\nabla g_\alpha(x)\in\R^{n\times|\alpha|}$ and $\nabla h(x)\in\R^{n\times\ell}$ are the matrices whose columns are the gradients $\nabla g_i(x)$, $i\in\alpha$, and $\nabla h_j(x)$, $j\in[\ell]$, respectively. We set
\begin{equation}\label{r_alpha}
r_\alpha:=\min_{z\in Z_\alpha}\rank(J_{\Phi_\alpha}(z)) \ \text{ if } Z_\alpha\neq\emptyset, \qquad r_\alpha:=+\infty \ \text{ if } Z_\alpha=\emptyset,
\end{equation}
so that $N_\alpha-r_\alpha=-\infty$ when $Z_\alpha=\emptyset$.

\begin{remark}\label{rmk:Dphi}
Clearly, $r_\alpha\geq\min_{z\in\R^{N_\alpha}}\rank(J_{\Phi_\alpha}(z))$. The minimum on the right-hand side is taken over the whole space $\R^{N_\alpha}$, so its computation does not require any knowledge of $Z_\alpha$; it may replace $r_\alpha$ in all the bounds below, at the price of a possibly weaker bound. The two numbers coincide if $Z_\alpha\neq\emptyset$ and the rank of $J_{\Phi_\alpha}$ is constant on $\R^{N_\alpha}$, as in Example~\ref{ex:1} and Remark~\ref{rmk:aff} below.
\end{remark}

The following theorem is the main result of this section.

\begin{theorem}\label{thm:dim1}
Assume that the ACQ holds on $K$.
\begin{enumerate}
\item[\rm (a)] For every $\alpha \subset [m]$,
\begin{equation}\label{dim_al}
\dim(\Sol(K,F)\cap K_{\alpha})\leq\min\left\lbrace \dim(K_\alpha),\ N_\alpha-r_\alpha \right\rbrace.
\end{equation}
Consequently,
\begin{equation}\label{dim_Sol}
\dim(\Sol(K,F))\leq\max_{\alpha\subset [m]}\ \min\left\lbrace \dim(K_\alpha),\ N_\alpha-r_\alpha \right\rbrace.
\end{equation}
\item[\rm (b)] Suppose that, for every $\alpha\subset[m]$, the matrix $J_{\Phi_\alpha}(z)$ is nonsingular for every $z\in Z_\alpha$ (this holds, in particular, if $J_{\Phi_\alpha}(z)$ is nonsingular for every $z\in\R^{N_\alpha}$). Then $\Sol(K,F)$ is finite.
\end{enumerate}
\end{theorem}

\begin{proof}
(a) Fix $\alpha\subset[m]$. If $\Sol(K,F)\cap K_\alpha=\emptyset$ there is nothing to prove, so assume that this set is nonempty. Since $\Sol(K,F)\cap K_{\alpha}\subset K_{\alpha}$, (D1) gives $\dim(\Sol(K,F)\cap K_{\alpha})\leq \dim(K_\alpha)$.

Let $x\in\Sol(K,F)\cap K_{\alpha}$. Since the ACQ holds on $K$, there exist $\lambda\in\R^m$ and $\mu\in\R^\ell$ satisfying \eqref{KKT}. Since $\lambda\geq0$, $g(x)\leq0$, and $\lambda^Tg(x)=0$, each product $\lambda_ig_i(x)$ vanishes. As $x\in K_\alpha$, we have $g_i(x)<0$ for $i\notin\alpha$, and therefore $\lambda_i=0$ for all $i\in[m]\setminus\alpha$. Hence the first line of \eqref{KKT} reduces to $F(x)+\sum_{i\in\alpha}\lambda_i\nabla g_i(x)+\sum_{j\in[\ell]}\mu_j\nabla h_j(x)=0$; together with $g_\alpha(x)=0$ and $h(x)=0$, this means that $(x,\lambda_\alpha,\mu)\in Z_\alpha$. In particular $Z_\alpha\neq\emptyset$, and we have shown that
$$\Sol(K,F)\cap K_{\alpha}\subset\pi(Z_\alpha),$$
where $\pi:\R^{N_\alpha}\to\R^n$ is the projection $\pi(x,\lambda_\alpha,\mu)=x$. The set $\pi(Z_\alpha)$ is semialgebraic by the Tarski--Seidenberg theorem, and (D1), (D3), and Lemma~\ref{lem:dim} with $W=Z_\alpha$ give
$$\dim(\Sol(K,F)\cap K_{\alpha})\leq\dim(\pi(Z_\alpha))\leq\dim(Z_\alpha)\leq N_\alpha-r_\alpha.$$
This proves \eqref{dim_al}. Substituting \eqref{dim_al} into \eqref{dimsum1} yields \eqref{dim_Sol}.

(b) Under this assumption, $r_\alpha=N_\alpha$ for every $\alpha$ with $Z_\alpha\neq\emptyset$, so each term on the right-hand side of \eqref{dim_Sol} is either at most $N_\alpha-r_\alpha=0$ or equal to $-\infty$. Hence $\dim(\Sol(K,F))\leq 0$, that is, $\Sol(K,F)$ is finite.
\end{proof}

The next example shows that the bound \eqref{dim_Sol} can be attained.

\begin{example}\label{ex:1}
Let $n=2$, $m=1$, and $\ell=0$, and consider $\PVI(K,F)$ with
$$F_1(x_1,x_2)= F_2(x_1,x_2)=x_1^{3}+x_1-x_2,\qquad K=\{(x_1,x_2)\in\R^2: x_1+x_2\leq 0\},$$
so that $g_1(x)=x_1+x_2$. Since $g_1$ is affine, the ACQ holds on $K$. The Jacobian matrix
$$J_F(x)=\begin{bmatrix}
	3x_1^2+1& -1 \\
    3x_1^2+1& -1
	\end{bmatrix}$$
has rank $1$ for every $x\in\R^2$. For $\alpha=\emptyset$ we have $K_{\emptyset}=\{x: x_1+x_2< 0\}$, $\dim(K_\emptyset)=2$, and $J_{\Phi_\emptyset}=J_F$, so $r_\emptyset=1$. For $\alpha=\{1\}$ we have $K_{\{1\}}=\{x: x_1+x_2= 0\}$, $\dim(K_{\{1\}})=1$, and
$$J_{\Phi_{\{1\}}}=\begin{bmatrix}
	3x_1^2+1& -1 &1\\
	3x_1^2+1& -1 &1 \\
    1 & 1 & 0
    \end{bmatrix},$$
whose rank is $2$ for all $(x,\lambda)\in\R^2\times\R$, because its first two rows coincide while its first and third rows are linearly independent; thus $r_{\{1\}}=2$. Theorem~\ref{thm:dim1}(a) gives
$$\dim(\Sol(K,F))\leq\max\{\min\{2,2-1\},\ \min\{1,3-2\}\}=1.$$
On the other hand, $K_\emptyset$ is the interior of $K$, so a point $x\in K_\emptyset$ is a solution if and only if $F(x)=0$, i.e., $x_2=x_1^3+x_1$ and $x_1+x_2<0$, which amounts to $x_2=x_1^3+x_1$ with $x_1<0$. A point $x=(t,-t)\in K_{\{1\}}$ is a solution if and only if $F(x)\in -N_K(x)=\{(s,s): s\leq 0\}$; since $F(t,-t)=(t^3+2t)(1,1)$, this means $t^3+2t\leq0$, i.e., $t\leq0$. Therefore
$$\Sol(K,F)=\{(x_1,x_1^3+x_1): x_1<0\}\cup\{(t,-t): t\leq 0\},$$
and $\dim(\Sol(K,F))=1$. Thus the bound \eqref{dim_Sol} is attained.
\end{example}

\begin{remark}\label{rmk:aff}
Let $F(x)=Mx+q$ with $M\in\R^{n\times n}$ and $q\in\R^n$, and let $K=\{x\in\R^n: Ax\leq b\}$ be polyhedral, where $A\in\R^{m\times n}$ and $b\in\R^m$; then $\PVI(K,F)$ is an affine variational inequality \cite{LTY05}. Here $\ell=0$ and $g_i(x)=a_i^Tx-b_i$, where $a_i^T$ is the $i$th row of $A$; the ACQ holds on $K$, and every matrix $J_{\Phi_\alpha}$ is constant:
$$J_{\Phi_\alpha}=\begin{bmatrix} M & A_\alpha^T\\ A_\alpha & 0\end{bmatrix},$$
where $A_\alpha$ is the submatrix of $A$ formed by the rows indexed by $\alpha$. Hence $r_\alpha$ is the rank of this matrix whenever $Z_\alpha\neq\emptyset$, and Theorem~\ref{thm:dim1}(a) gives
$$\dim(\Sol(K,F))\leq\max_{\alpha\subset [m]}\ \min\left\lbrace \dim(K_\alpha),\ n+|\alpha|-\rank\begin{bmatrix} M & A_\alpha^T\\ A_\alpha & 0\end{bmatrix} \right\rbrace.$$
\end{remark}

\section{Polynomial complementarity problems}\label{sec:4}

In this section we apply Theorem~\ref{thm:dim1} to the nonnegative orthant $K=\R^n_+=\{x\in\R^n: -x_i\leq 0,\ i\in[n]\}$; here $m=n$, $\ell=0$, and $g_i(x)=-x_i$. The problem $\PVI(\R^n_+,F)$ is the \textit{polynomial complementarity problem} \cite{Gowda16}, denoted by $\PCP(F)$: find $x\in\R^n$ such that
$$x\geq0,\qquad F(x)\geq0,\qquad x^TF(x)=0.$$
Since the functions $g_i$ are affine, the ACQ holds on $\R^n_+$. For $\alpha\subset[n]$ the pseudo-face is
$$K_{\alpha}=\{x\in\R^n: x_i=0\ (i\in\alpha),\ x_j>0\ (j\in [n]\setminus\alpha)\}.$$
Let $C_\alpha\in\R^{n\times|\alpha|}$ be the matrix whose columns are the standard unit vectors $e_j$, $j\in\alpha$. Then $\nabla g_\alpha(x)=-C_\alpha$, the Hessians $\nabla^2g_i$ vanish, and by \eqref{Dphi}
\begin{equation}\label{Dphi1}
J_{\Phi_\alpha}(x)=\begin{bmatrix}
J_F(x) &  -C_{\alpha}\\
-C^T_{\alpha} & 0
\end{bmatrix},
\end{equation}
which does not depend on $\lambda_\alpha$. For $\beta\subset[n]$ and $B\in\R^{n\times n}$, we denote by $B_{\beta\beta}$ the principal submatrix of $B$ with rows and columns indexed by $\beta$, with the convention $\rank(B_{\emptyset\emptyset})=0$. We also write $\alpha^c:=[n]\setminus\alpha$.

\begin{lemma}\label{lem:pcp}
For every $x\in\R^n$ and every $\alpha\subset[n]$,
\begin{equation}\label{rankPhi}
\rank(J_{\Phi_\alpha}(x))=2|\alpha|+\rank\big(J_F(x)_{\alpha^c\alpha^c}\big).
\end{equation}
\end{lemma}

\begin{proof}
Write $B:=J_F(x)$ and $s:=|\alpha|$. If $\alpha=\emptyset$, then $J_{\Phi_\alpha}(x)=B$ and there is nothing to prove, so we assume that $\alpha\neq\emptyset$. A simultaneous permutation of the rows and columns of $J_{\Phi_\alpha}(x)$, which does not change its rank, moves the coordinates indexed by $\alpha^c$ in front of those indexed by $\alpha$. After this permutation $C_\alpha=\begin{bmatrix}0\\ I_s\end{bmatrix}$, and $J_{\Phi_\alpha}(x)$ becomes
$$\mathcal{M}=\begin{bmatrix}
B_{\alpha^c\alpha^c} & B_{\alpha^c\alpha} & 0\\
B_{\alpha\alpha^c} & B_{\alpha\alpha} & -I_s\\
0 & -I_s & 0
\end{bmatrix}
=\begin{bmatrix} B_{\alpha^c\alpha^c} & R\\ L & T\end{bmatrix},
\qquad
T:=\begin{bmatrix} B_{\alpha\alpha} & -I_s\\ -I_s & 0\end{bmatrix},$$
with $R=\begin{bmatrix}B_{\alpha^c\alpha} & 0\end{bmatrix}$ and $L=\begin{bmatrix}B_{\alpha\alpha^c}\\ 0\end{bmatrix}$. A direct computation shows that the block $T$ is invertible, with
$$T^{-1}=\begin{bmatrix} 0 & -I_s\\ -I_s & -B_{\alpha\alpha}\end{bmatrix};$$
in particular, $\rank(T)=2s$. If $\alpha=[n]$, then the blocks $B_{\alpha^c\alpha^c}$, $R$, and $L$ are absent, $\mathcal{M}=T$, and \eqref{rankPhi} follows. Otherwise, the block factorization (see, e.g., \cite[Chapter 2]{Zhang})
$$\mathcal{M}=\begin{bmatrix} I & RT^{-1}\\ 0 & I\end{bmatrix}
\begin{bmatrix} B_{\alpha^c\alpha^c}-RT^{-1}L & 0\\ 0 & T\end{bmatrix}
\begin{bmatrix} I & 0\\ T^{-1}L & I\end{bmatrix}$$
expresses $\mathcal{M}$ as the product of a block diagonal matrix and two block triangular matrices whose diagonal blocks are identity matrices; the latter two matrices are invertible. A direct computation gives $RT^{-1}=\begin{bmatrix}0 & -B_{\alpha^c\alpha}\end{bmatrix}$, hence $RT^{-1}L=0$. Therefore
$$\rank(\mathcal{M})=\rank(B_{\alpha^c\alpha^c})+\rank(T)=\rank(B_{\alpha^c\alpha^c})+2s,$$
which is \eqref{rankPhi}.
\end{proof}

Recall that a square matrix is \textit{nondegenerate} \cite[Definition 3.6.1]{CPS1992} if all its principal minors are nonzero.

\begin{corollary}\label{thm:dim2}
Consider the polynomial complementarity problem $\PCP(F)$ and, for $\beta\subset[n]$, put
$$s_\beta:=\min_{x\in\R^n_+}\rank\big(J_F(x)_{\beta\beta}\big).$$
Then
\begin{equation}\label{dim_pcp}
\dim(\Sol(\R^n_+,F))\leq\max_{\beta\subset[n]}\big(|\beta|-s_\beta\big).
\end{equation}
In particular, if $J_F(x)$ is nondegenerate for every $x\in\R^n_+$, then $\Sol(\R^n_+,F)$ is finite.
\end{corollary}

\begin{proof}
Fix $\alpha\subset[n]$. If $\Sol(\R^n_+,F)\cap K_\alpha=\emptyset$ there is nothing to prove, so assume that this set is nonempty, and put
$$W_\alpha:=\{(x,\lambda_\alpha)\in Z_\alpha:\ x\in K_\alpha\}.$$
The proof of Theorem~\ref{thm:dim1} shows that $\Sol(\R^n_+,F)\cap K_\alpha\subset\pi(W_\alpha)$, where $\pi(x,\lambda_\alpha)=x$; in particular, $W_\alpha$ is a nonempty semialgebraic set. Every $z=(x,\lambda_\alpha)\in W_\alpha$ satisfies $x\in K_\alpha\subset\R^n_+$, so Lemma~\ref{lem:pcp} gives $\rank(J_{\Phi_\alpha}(z))=2|\alpha|+\rank(J_F(x)_{\alpha^c\alpha^c})\geq 2|\alpha|+s_{\alpha^c}$. By (D1), (D3), and Lemma~\ref{lem:dim} applied with $W=W_\alpha$ and $N=N_\alpha=n+|\alpha|$,
$$\dim(\Sol(\R^n_+,F)\cap K_\alpha)\leq\dim(\pi(W_\alpha))\leq\dim(W_\alpha)\leq N_\alpha-2|\alpha|-s_{\alpha^c}=|\alpha^c|-s_{\alpha^c}.$$
Taking the maximum over $\alpha$, i.e., over $\beta=\alpha^c$, and using \eqref{dimsum1}, we obtain \eqref{dim_pcp}. If $J_F(x)$ is nondegenerate for every $x\in\R^n_+$, then $s_\beta=|\beta|$ for all $\beta$, the right-hand side of \eqref{dim_pcp} is $0$, and the solution set is finite.
\end{proof}

Note that Corollary~\ref{thm:dim2} requires the nondegeneracy of $J_F(x)$ only for $x\in\R^n_+$. For instance, if $n=1$ and $F(x)=(x+1)^2$, then $J_F(x)=2(x+1)$ vanishes at $x=-1$ but not on $\R_+$, so Corollary~\ref{thm:dim2} shows that $\Sol(\R_+,F)$ is finite; indeed, $\Sol(\R_+,F)=\{0\}$.

\begin{example}\label{ex:pcp}
Consider $\PCP(F)$ with
$$F_1(x_1,x_2)=x_1-1, \qquad F_2(x_1,x_2)=-x_1^{3}-x_2+1.$$
The Jacobian matrix
$$J_F(x_1,x_2)=\begin{bmatrix}
	1& 0 \\
	-{3}x_1^2 & -1
	\end{bmatrix}$$
has the $1\times1$ principal minors $1$ and $-1$ and the determinant $\det(J_F(x))=-1$. Thus $J_F(x)$ is nondegenerate for every $x\in\R^2$, and Corollary~\ref{thm:dim2} shows that $\Sol(\R^2_+,F)$ is finite. Indeed, since $x\geq0$ and $F(x)\geq0$, the condition $x^TF(x)=0$ means that $x_iF_i(x)=0$ for $i=1,2$. Now $F_1(x)\geq0$ forces $x_1\geq1$; then $x_1F_1(x)=0$ forces $x_1=1$, and $F_2(x)=-x_2\geq0$ forces $x_2=0$. Hence
$$\Sol(\R^2_+,F)=\{(1,0)\}.$$
\end{example}

When $F(x)=Mx+q$ with $M\in\R^{n\times n}$ and $q\in\R^n$, the problem $\PCP(F)$ is the linear complementarity problem $\LCP(q,M)$ \cite{CPS1992}, and $J_F\equiv M$. In this case Corollary~\ref{thm:dim2} shows that $\LCP(q,M)$ has finitely many solutions for every $q\in\R^n$ whenever $M$ is nondegenerate; this is a classical result, and its converse also holds \cite[Theorem 3.6.3]{CPS1992}. The following remark shows that the nonsingularity of $J_F(x)$ for every $x$ does not by itself imply finiteness, and that the bound \eqref{dim_pcp} can be attained.

\begin{remark}\label{rmk:nonsing}
Let $n=2$ and $F(x)=Mx$, that is, $q=0$ and $M=\begin{bmatrix}0&1\\1&0\end{bmatrix}$. The matrix $M$ is nonsingular, but its principal minors $M_{11}=M_{22}=0$ vanish. The conditions $x\geq0$, $Mx=(x_2,x_1)^T\geq0$, and $x^TMx=2x_1x_2=0$ give
$$\Sol(\R^2_+,F)=\{(t,0):t\geq0\}\cup\{(0,t):t\geq0\},$$
an infinite set of dimension $1$. Here $s_\beta=\rank(M_{\beta\beta})$, so the right-hand side of \eqref{dim_pcp} equals $\max\{0,\ 1-0,\ 1-0,\ 2-2\}=1$, and \eqref{dim_pcp} holds with equality. By Lemma~\ref{lem:pcp}, $\rank(J_{\Phi_\alpha})=2<3=N_\alpha$ for $\alpha=\{1\}$ and for $\alpha=\{2\}$, because $M_{22}=0$ and $M_{11}=0$, respectively. Accordingly, the solution set meets the pseudo-faces $K_{\{1\}}=\{x:x_1=0,\ x_2>0\}$ and $K_{\{2\}}=\{x:x_2=0,\ x_1>0\}$ in the one-dimensional sets $\{(0,t):t>0\}$ and $\{(t,0):t>0\}$.
\end{remark}

\begin{remark}[Stationary points of polynomial fractional programs]\label{rmk:frac}
Let $p$ and $q$ be polynomials with $q>0$ on $K$, and let $f:=p/q$. A point $x\in K$ is a \textit{stationary point} of the polynomial fractional program \cite{TTK04}
$${\rm minimize} \ f(x) \ {\rm subject \ to } \ x\in K$$
if $\langle\nabla f(x),y-x\rangle\geq0$ for all $y\in K$ (see, e.g., \cite[Subsection 1.3.1]{FaPa03}); every local minimizer of $f$ on $K$ is such a point. Since $\nabla f=G/q^2$ with $G:=q\nabla p-p\nabla q$, and $q^2>0$ on $K$, the set $\Stat(K,f)$ of stationary points is equal to $\Sol(K,G)$. Hence Theorem~\ref{thm:dim1} and Corollary~\ref{thm:dim2} apply to $\Stat(K,f)$, with $J_F$ replaced by
$$J_G=q\nabla^2p-p\nabla^2q+\nabla p\,\nabla q^T-\nabla q\,\nabla p^T.$$
If, moreover, $f$ is pseudoconvex on $K$, then $\Stat(K,f)$ is the set of global minimizers of $f$ on $K$.
\end{remark}

\subsection*{Acknowledgements}
The author used Claude (Anthropic), an AI assistant, to restructure and shorten the exposition and to check the proofs and examples.

\bibliographystyle{amsplain}  
\bibliography{Dim_PVIs}

\end{document}